\RequirePackage{fix-cm}
\documentclass{svjour3_no_journal_info}                     
\smartqed  
\usepackage[english]{babel}
\usepackage[utf8]{inputenc}
\usepackage[T1]{fontenc}
\usepackage{tikz}
\usepackage{color}
\usepackage{tikz}
\usepackage{bbm}
\usepackage{algorithm}
\usepackage{algpseudocode}
\usepackage{float}
\usepackage{tabularx}

\usetikzlibrary{arrows,calc,positioning}

\usepackage[letterpaper,top=3cm,bottom=3cm,left=3cm,right=3cm,marginparwidth=1.75cm]{geometry}

\usepackage{amsmath,amssymb}
\usepackage{graphicx}
\usepackage[colorinlistoftodos]{todonotes}
\usepackage[colorlinks=true, allcolors=blue]{hyperref}
\usepackage[T1]{fontenc}
\usepackage{tikz}
\usetikzlibrary{arrows,calc,positioning}

\tikzstyle{level12}=[draw,text centered,minimum size=2em,text width=1.8cm,text height=0.34cm]
\tikzstyle{cases}=[draw,text centered,minimum size=2em,text width=2.75cm,text height=0.34cm]
\tikzstyle{finite}=[draw,text centered,minimum size=2em,text width=4cm,text height=4.9cm, dashed]

\DeclareMathOperator{\Prox}{Prox}
\DeclareMathOperator*{\argmin}{arg\,min}
\DeclareMathOperator*{\argmax}{arg\,max}
\newcommand{\relint}{\mathbf{relint}}

\newcommand{\ran}{\mathbf{ran}}

\begin{document}

\title{A New Use of Douglas-Rachford Splitting and ADMM
 for Identifying Infeasible, Unbounded, and Pathological Conic Programs
\thanks{This work is supported in part by NSF grant DMS-1720237 and ONR grant N000141712162.}
}


\titlerunning{A New Use of DRS and ADMM
 for Identifying Infeasible, Unbounded, and Pathological Conic Programs}        

\author{Yanli Liu         \and
        Ernest K. Ryu     \and
        Wotao Yin
}


\institute{Yanli Liu \and
        Ernest K. Ryu     \and
        Wotao Yin\at
              Department of Mathematics, UCLA, Los Angeles, CA 90095-1555. \\
              \email{yanli\,/\,eryu\,/\,wotaoyin@math.ucla.edu}           
           \and
           \and
}

\date{\today}

\maketitle


\begin{abstract}
In this paper, we present a method for identifying
infeasible, unbounded, and pathological conic programs based on Douglas-Rachford splitting,
or equivalently ADMM.
When an optimization program is infeasible, unbounded, or pathological,
the iterates of Douglas-Rachford splitting diverge.
Somewhat surprisingly, such divergent iterates
still provide useful information, which our method
uses for identification.
In addition, for strongly
infeasible problems
the method produces a separating hyperplane
and informs the user on how to minimally
modify the given problem
to achieve strong feasibility.
As a first-order method, the proposed algorithm
relies on simple subroutines,
and therefore is simple to implement and has
low per-iteration cost.
\end{abstract}

\section{Introduction}
Many convex optimization algorithms
have strong theoretical guarantees and empirical performance, but they are often limited to non-pathological, feasible problems;
under pathologies often the theory breaks down
and the empirical performance degrades significantly.
In fact, the behavior of convex optimization algorithms
under pathologies has been studied much less,
and many existing solvers often simply report ``failure''
without informing the users of what went wrong
upon encountering infeasibility, unboundedness, or pathology.
Pathological problem are numerically challenging,
but they are not impossible to deal with.
As infeasibility, unboundedness, and pathology do arise in practice (see, for example, \cite{Lofberg2009_pre,LoeraMalkinParrilo2012_computation}),
designing a robust algorithm that behaves well in all cases
is important to the completion of a robust solver.

In this paper, we propose a method based on Douglas-Rachford splitting (DRS),
or equivalently ADMM, 
that identifies infeasible, unbounded, and pathological
conic programs. First-order methods such as DRS/ADMM
are simple and can quickly provide a solution with moderate accuracy.
It is well known, for example, by combining Theorem 1 of \cite{Rockafellar1976_monotone} and Proposition 4.4 of \cite{Eckstein1989_splitting}, that the iterates of DRS/ADMM converge to a fixed point if there is one
(a fixed point $z^*$ of an operator $T$ satisfies $z^*=Tz^*$),
and when there is no fixed point, 
the iterates diverge unboundedly.
However, the precise manner in which they diverge has been studied much less.
Somewhat surprisingly, when iterates of DRS/ADMM diverge,
the behavior of the iterates still provides useful information, which we use to classify the conic program. For example, a separating hyperplane can be found when the conic program is strongly infeasible, and an improving direction can be obtained when there is one.
When the problem is infeasible or weakly feasible, it is useful to know
how to minimally modify the problem data to achieve strong feasibility.
We also get this information via the divergent iterates.

Facial reduction is one approach to handle infeasible or pathological conic programs.
Facial reduction reduces an infeasible or pathological problem
into a new problem that is strongly feasible, strongly infeasible, or unbounded with an improving direction,
which are the easier cases \cite{BorweinWolkowicz1981_facial,BorweinWolkowicz1981_regularizing,Pataki2000_simple,WakiMuramatsu2013_facial}.

Many existing methods such as interior point methods or homogeneous self-dual embedding \cite{LuoSturmZhang2000_conic,Yoshise2012_complementarity}
cannot directly handle certain pathologies, such as weakly feasible or weakly infeasible problems,
and are forced to use facial reduction \cite{LourencoMuramatsuTsuchiya2015_solvinga,PermenterFribergAndersen2015_solving}.
However, facial reduction introduces a new set of computational issues.
After completing the facial reduction step, which has its own the computational challenge and cost,
the reduced problem must be solved.
The reduced problem involves a cone 
expressed as an intersection of the original cone with an linear subspace,
and in general such cones neither are self-dual nor have a simple formula for projection.
This makes applying an interior point method or a first-order method difficult,
and existing work on facial reduction do not provide an efficient way to address this issue.

In contrast, our proposed method directly address infeasibility, unboundedness, and pathology.
Some cases are always identified, and some are identifiable under certain conditions.
Being a first-order method, the proposed algorithm relies on simple subroutines;
each iteration performs projections onto the cone and the affine
space of the conic program and 
elementary operations such as vector addition.
Consequently, the method is simple to implement and has a lower per-iteration cost than interior point methods.

\subsection{Basic definitions}\label{Sec.  2.1}
\paragraph{Cones.}
A set $K\subseteq \mathbb{R}^n$ is a cone if $K=\lambda K$ for any $\lambda>0$. 
We write and define the dual cone of $K$ as
\[
K^*=\{u\in \mathbb{R}^n|\,\,u^Tv\geq 0,\,\, \text{for all}\,\, v \in K\}.
\]

Throughout this paper, we will focus on nonempty closed convex cones 
that we can efficiently project onto.
In particular, we do \emph{not} require that the cone be self-dual.
Example of such cones include:
\begin{itemize}
\item
The positive orthant: 
\[
\mathbb{R}_+^k=\{x\in \mathbb{R}^k\,|\,x_i\ge 0,\,i=1,\dots,n\}
\]
\item
Second order cone:
$$Q^{k+1}=\left\{(x_1,\dots,x_k, x_{k+1})\in \mathbb{R}^k\times \mathbb{R}_{+}\,|\,x_{k+1}\geq \sqrt{x_1^2+\dots+x_k^2}\right\}$$
 \item
 Rotated second order cone:
 $$Q_r^{k+2}=\left\{(x_1,\dots,x_k, x_{k+1}, x_{k+2})\in\mathbb{R}^k\times \mathbb{R}^2_{+}\,|\,2x_{k+1}x_{k+2}\geq x_1^2+\dots+x_k^2\right\}.$$
 \item Positive semidefinite cone:
 $$S^k_{+}=\{M=M^T\in \mathbb{R}^{k\times k}|\,\, x^TMx\geq 0\,\, \text{for any}\,\, x\in \mathbb{R}^k\}$$

\end{itemize}

\paragraph{Conic programs.}
Consider the conic program
\begin{equation}
\begin{array}{ll}
\mbox{minimize}& c^Tx\\
\mbox{subject to}& Ax=b \\
\mbox{} &          x\in K,
\end{array}
\tag{P}\label{P}
\end{equation}
where $x\in \mathbb{R}^n$ is the optimization variable,
$c\in \mathbb{R}^n$, $A \in \mathbb{R}^{m\times n}$, and $b\in \mathbb{R}^m$
are problem data, and $K\subseteq \mathbb{R}^n$ is a nonempty closed convex cone.
We write $p^\star=\inf\{c^Tx\,|\,Ax=b,\, x\in K\}$ to denote the optimal value of \eqref{P}.
For simplicity, we assume $m\le n$ and $A$ is full rank.

The dual problem of \eqref{P} is
\begin{equation}
\begin{array}{ll}
\mbox{maximize}& b^Ty\\
\mbox{subject to}& A^Ty+s=c\\
\mbox{}& s\in K^*,
\end{array}
\tag{D}\label{D}
\end{equation}
where $y\in \mathbb{R}^m$ and $s\in \mathbb{R}^n$ are the optimization variables.
We write $d^\star=\sup\{b^Ty\,|\,A^Ty+s=c,\, s\in K^*\}$ to denote the optimal value of \eqref{D}.

The optimization problem \eqref{P} is either feasible or infeasible;
\eqref{P} is feasible if there is an $x\in K\cap \{x\,|\,Ax=b\}$
and infeasible if there is not.
When \eqref{P} is feasible, it is 
strongly feasible if there is an
$x\in \relint K\cap \{x\,|\,Ax=b\}$
and weakly feasible if there is not,
where $\relint$ denotes the relative interior.
When \eqref{P} is infeasible, it is strongly
infeasible if there is a non-zero distance between $K$ and $\{x\,|\,Ax=b\}$,
i.e., $d(K,\{x\,|\,Ax=b\})>0$,
and weakly infeasible if $d(K,\{x\,|\,Ax=b\})=0$,
where
\[
d(C_1, C_2)=\inf
\left\{\|x-y\|
\,|\,
x\in C_1,\, y\in C_2
\right\},
\]
and $\|\cdot\|$ denotes the Euclidean norm.
Note that $d(C_1, C_2)=0$ does not necessarily imply $C_1$ and $C_2$
intersect.
When \eqref{P} is infeasible, we say $p^\star=\infty$,
and when feasible,  $p^\star\in \mathbb{R}\cup \{-\infty\}$.
Likewise, 
when \eqref{D} is infeasible, we say $d^\star=-\infty$,
and when  feasible,  $d^\star\in \mathbb{R}\cup \{\infty\}$.

As special cases, 
\eqref{P} is called a linear program
when $K$ is the positive orthant,
a second-order cone program when $K$
is the second-order cone,
and 
a semidefinite program 
when $K$ is the positive semidefinite cone.

\subsection{Classification of conic programs}
Every conic program of the form \eqref{P} falls under exactly one of the following $7$ cases (some of the following examples are taken from \cite{LuoSturmZhang2000_conic,LuoSturmZhang1997_duality,LourencoMuramatsuTsuchiya2015_solvinga,LuenbergerYe2016_conic}). Discussions on most of these cases exist in the literature.
Some of these cases have a corresponding dual characterization,
but we skip this discussion as it is not directly relevant to our method. We report the results of SDPT3, SeDuMi, and MOSEK using their default settings.
In Section~\ref{s:theory}, we discuss how to identify most of these $7$ cases.

\paragraph{Case (a).}
$p^{\star}$ is finite, both \eqref{P} and \eqref{D} have solutions, and $d^\star=p^\star$,
which is the most common case.
For example, the problem
\begin{equation*}
\begin{array}{ll}
\mbox{minimize}& x_3\\
\mbox{subject to}& x_1=1 \\
&x_3\ge \sqrt{x_1^2+x_2^2}
\end{array}
\end{equation*}
has the solution $x^{\star}=(1,0,1)$ and $p^\star=1$.
(The inequality constraint corresponds to $x\in Q^3$.)
SDPT3, SeDuMi and MOSEK can solve this example.

The dual problem, after some simplification, is
\[
\begin{array}{ll}
\mbox{maximize}& y\\
\mbox{subject to} &1\ge y^2,
\end{array}
\]
which has the solution $y^\star=1$ and $d^\star=1$.

\paragraph{Case (b).}
$p^{\star}$ is finite, \eqref{P} has a solution, but \eqref{D} has no solution, or 
$d^\star<p^\star$, or both.
For example, the problem
\begin{equation*}
\begin{array}{ll}
\mbox{minimize}& x_2\\
\mbox{subject to}& x_1=x_3=1\\
&x_3\ge \sqrt{x_1^2+x_2^2}
\end{array}
\end{equation*}
has the solution $x^{\star}=(1,0,1)$ and optimal value $p^\star=0$.
(The inequality constraint corresponds to $x\in Q^3$.)

In this example, SDPT3 reports ``Inaccurate/Solved'' and $-2.99305\times 10^{-5}$ as the optimal value; SeDuMi reports ``Solved'' and $-1.54566\times 10^{-4}$ as the optimal value; MOSEK reports ``Solved'' and $-2.71919\times 10^{-8}$ as the optimal value.

The dual problem, after some simplification, is
\[
\begin{array}{ll}
\mbox{maximize}& y_1-\sqrt{1+y_1^2}.
\end{array}
\]
By taking $y_1\rightarrow \infty$ we achieve the dual optimal value $d^\star=0$,
but no finite $y_1$ achieves it.

As another example, the problem
\[
\begin{array}{ll}
\mbox{minimize}& 2x_{12}\\
\mbox{subject to}&
X=
\begin{bmatrix}
x_{11}& x_{12} & x_{13}\\
x_{12}& 0 & x_{23}\\
x_{13}& x_{23} & x_{12}+1
\end{bmatrix}
\in S_+^3,
\end{array}
\]
has the solution
\[
X^\star=
\begin{bmatrix}
0&0&0\\
0&0&0\\
0&0&1
\end{bmatrix}
\]
and optimal value $p^\star=0$.

The dual problem, after some simplification, is
\[
\begin{array}{ll}
\mbox{maximize}& 2y_2\\
\mbox{subject to}&
\begin{bmatrix}
0& y_2+1& 0\\
y_2+1& -y_1& 0\\
0&0&-2y_2
\end{bmatrix}
\in S_+^3,
\end{array}
\]
which has the solution $y^\star=(0,-1)$ and optimal value $d^\star=-2$.

In this SDP example, SDPT3 reports ``Solved'' and $-2$ as the optimal value; SeDuMi reports ``Solved'' and $-0.602351$ as the optimal value; MOSEK reports ``Failed'' and does not report an optimal value.

Note that case (b) can happen only when \eqref{P} is weakly feasible,
by standard convex duality \cite{Rockafellar1974_conjugate}.

\paragraph{Case (c).}
\eqref{P} is feasible, $p^{\star}$ is finite, 
but there is no solution. For example, the problem
\begin{equation*}
\begin{array}{ll}
\mbox{minimize}& x_3\\
\mbox{subject to}& x_1=\sqrt{2} \\
&2x_2x_3\ge x_1^2\\
&x_2, x_3\ge 0
\end{array}
\end{equation*}
has an optimal value $p^{\star}=0$
but has no solution since any feasible $x$ satisfies $x_3> 0$.
(The inequality constraints correspond to $x\in Q_r^3$.)

In this example, SDPT3 reports ``Inaccurate/Solved'' and $7.9509\times 10^{-5}$ as the optimal value; SeDuMi reports ``Solved'' and $8.75436\times 10^{-5}$ as the optimal value; MOSEK reports ``Solved'' and $4.07385\times 10^{-8}$ as the optimal value.

\paragraph{Case (d).}
\eqref{P} is feasible, $p^\star=-\infty$,
and there is an improving direction, 
i.e., there is a $u\in \mathcal{N}(A)\cap K$ satisfying $c^Tu<0$. For example, the problem
\begin{equation*}
\begin{array}{ll}
\mbox{minimize}& x_1\\
\mbox{subject to}& x_2=0\\ 
&x_3\ge \sqrt{x_1^2+x_2^2}
\end{array}
\end{equation*}
has an improving direction $u=(-1,0,1)$.
If $x$ is any feasible point, $x+tu$ is feasible for $t\ge 0$, and the objective value
goes to $-\infty$ as $t\rightarrow \infty$.
(The inequality constraint corresponds to $x\in Q^3$.)

In this example, SDPT3 reports ``Failed'' and does not report an optimal value; SeDuMi reports ``Unbounded'' and $-\infty$ as the optimal value; MOSEK reports ``Unbounded'' and $-\infty$ as the optimal value.

\paragraph{Case (e).}
\eqref{P} is feasible, $p^\star=-\infty$,
but there is no improving direction, 
i.e., there is no $u\in \mathcal{N}(A)\cap K$ satisfying $c^Tu<0$. 
For example, consider the problem
\begin{equation*}
\begin{array}{ll}
\mbox{minimize}& x_1\\
\mbox{subject to}& x_2=1 \\
&2x_2x_3\ge x_1^2\\
&x_2, x_3\ge 0.
\end{array}
\end{equation*}
(The inequality constraints correspond to $x\in Q_r^3$.)
Any improving direction $u=(u_1,u_2,u_3)$ would satisfy $u_2=0$,
and this in turn, with the cone constraint, implies $u_1=0$ and $c^Tu=0$. 
However, even though there is no improving direction,
we can eliminate the variables $x_1$ and $x_2$ to verify that
\[
p^{\star}=\inf\{-\sqrt{2x_3}\,|\, x_3\geq 0\}=-\infty.
\]

In this example, SDPT3 reports ``Failed'' and does not report an optimal value; SeDuMi reports ``Inaccurate/Solved'' and $-175514$ as the optimal value; MOSEK reports ``Inaccurate/Unbounded'' and $-\infty$ as the optimal value.

\paragraph{Case (f).} Strongly infeasible,
where $p^{\star}=\infty$ and $d(K,\{x\,|\,Ax=b\})>0$.
For example, the problem
\begin{equation*}
\begin{array}{ll}
\mbox{minimize}& 0\\
\mbox{subject to}& x_3=-1\\
&x_3\ge \sqrt{x_1^2+x_2^2}
\end{array}
\end{equation*}
satisfies $d(K,\{x\,|\,Ax=b\})=1$.
(The inequality constraint corresponds to $x\in Q^3$.)

In this example, SDPT3 reports ``Failed'' and does not report an optimal value; SeDuMi reports ``Infeasible'' and $\infty$ as the optimal value; MOSEK reports ``Infeasible'' and $\infty$ as the optimal value.

\paragraph{Case (g).} Weakly infeasible,
where $p^{\star}=\infty$ but $d(K,\{x\,|\,Ax=b\})=0$.
For example, the problem
\begin{equation*}
\begin{array}{ll}
\mbox{minimize}& 0\\
\mbox{subject to}& 
\begin{bmatrix}
0,1,1\\
1,0,0
\end{bmatrix}x=
\begin{bmatrix}
0\\
1
\end{bmatrix}\\ 
&x_3\ge \sqrt{x_1^2+x_2^2}
\end{array}
\end{equation*}
satisfies $d(K,\{x\,|\,Ax=b\})=0$,
since
\[
d(K,\{x\,|\,Ax=b\})\leq \|(1,-y,y)-(1,-y,\sqrt{y^2+1})\|\rightarrow 0
\]
as $y\rightarrow \infty$.
(The inequality constraint corresponds to $x\in Q^3$.)

In this example, SDPT3 reports ``Infeasible'' and $\infty$ as the optimal value; SeDuMi reports ``Solved'' and $0$ as the optimal value; MOSEK reports ``Failed'' and does not report an optimal value.

\paragraph{Remark.}
In the case of linear programming,
i.e., when $K$ in \eqref{P} is the positive orthant,
there are only three possible cases: (a), (d), and (f).

\subsection{Classification method overview}
At a high level, our proposed method for classifying the 7 cases is quite simple.
Given an operator $T$ and a starting point $z^0$, we call
$z^{k+1}=T(z^k)$ the \emph{fixed point iteration} of $T$.
Our proposed method runs
three similar but distinct fixed-point iterations with the operators
\begin{align}
T_1(z)&=\tilde{T}(z)+x_0-\gamma Dc\nonumber\\
T_2(z)&=\tilde{T}(z)+x_0 \label{Operators}\tag{Operators}\\
T_3(z)&=\tilde{T}(z)-\gamma Dc\nonumber,
\end{align}
where the common operator $\tilde{T}$ and the constants $D, \gamma, x_0$ are defined and explained in Section~\ref{s:theory} below.
We can view $T_1$ as the DRS operator of \eqref{P}, $T_2$ as the DRS operator with $c$ set to 0 in \eqref{P},
and $T_3$ as the DRS operator with $b$ set to 0 in \eqref{P}.
We use
the information provided by the iterates of these fixed-point iterations
to solve \eqref{P} and classify the cases,
based on the theory of Section~\ref{s:theory}
and the flowchart shown in Figure~\ref{Fig.1}
as outlined in Section~\ref{ss:alg} below.

\begin{figure}[!htb]
\centering
    \begin{tikzpicture}[>=latex', auto]
    \node [level12]  (start) {Start};
    
    \node [level12] (infea) [node distance=1cm and 3.5cm,above right=of start] {Infeasible};
    
    \node [level12] (fea) [node distance=1cm and 3.5cm,below right=of start] {Feasible};
    
    \node [cases] (f) [node distance=0.1cm and 3.5cm,above right=of infea] {(f) Strongly infeasible};
    
    \node [cases] (g) [node distance=0.1cm and 3.5cm,below right=of infea] {(g) Weakly infeasible};
    
    \node [cases] (a) [node distance=0.1cm and 3.5cm, right=of fea] 
    {(a) There is a primal-dual solution pair with $d^{\star}=p^{\star}$};
    
    \node[cases](b)[node distance = 1.4cm and 3.5cm, below right=of fea]
    {(b) There is a primal solution but no dual solution or $d^{\star}<p^{\star}$};
    
    \node [cases] (c) [node distance=4.3cm and 3.5cm,below right=of fea] {(c) $p^{\star}$ is finite but there is no solution};
    
    \node [finite] (fi) [node distance=1.0cm and 2.9cm,below right=of fea] {};
    
    \node [cases] (d) [node distance=6.6cm and 3.5cm,below right=of fea] {(d) Unbounded ($p^{\star}=-\infty$) with an improving direction};
    
    \node [cases] (e) [node distance=9.1cm and 3.5cm,below right=of fea] {(e) Unbounded ($p^{\star}=-\infty$) without an improving direction};

    \draw[line width=2pt,->] (start) -- node[above,pos=0.5,align=center] {\textbf{Thm 6}\\\textbf{Alg 2}}($(start.east)+(2,0)$) |- (infea) ;
    
    \draw[line width=2pt,->] (start) -- ($(start.east)+(2,0)$) |- (fea) ;
    
    \draw[line width=2pt,->] (infea) -- node[above,pos=0.5,align=center] {\textbf{Thm 7}\\\textbf{Alg 2}}($(infea.east)+(2,0)$) |- (f) ;
    
    \draw[line width=2pt,->] (infea) -- ($(infea.east)+(2,0)$) |- (g) ;
    
    \draw[line width=2pt,->] (fea)  --($(fea.east)+(0.3,0)$) |- node[above,pos=0.72,align=center] {\textbf{Thm 2}\\\textbf{Alg 1}} (a) ;
    
    \draw[line width=2pt,dashed,->] (fea) -- ($(fea.east)+(0.3,0)$) |- node[above,pos=0.75,align=center] {\textbf{Thm 11,12}\\\textbf{Alg 3}} (fi) ;
    
    \draw[line width=2pt,dashed,->] (fea) -- ($(fea.east)+(0.3,0)$) |- node[above,pos=0.75,align=center] {\textbf{Thm 13}\\\textbf{Alg 1}} (b) ;
    
    \draw[line width=2pt,->] (fea) -- ($(fea.east)+(0.3,0)$) |- node[above,pos=0.75,align=center] {\textbf{Thm 10}\\\textbf{Alg 3}} (d) ;
   
   \end{tikzpicture}
   \caption{The flowchart for identifying cases (a)--(g). A solid arrow means the cases are always identifiable, a dashed arrow means the cases sometimes identifiable.}
   \label{Fig.1}
\end{figure}
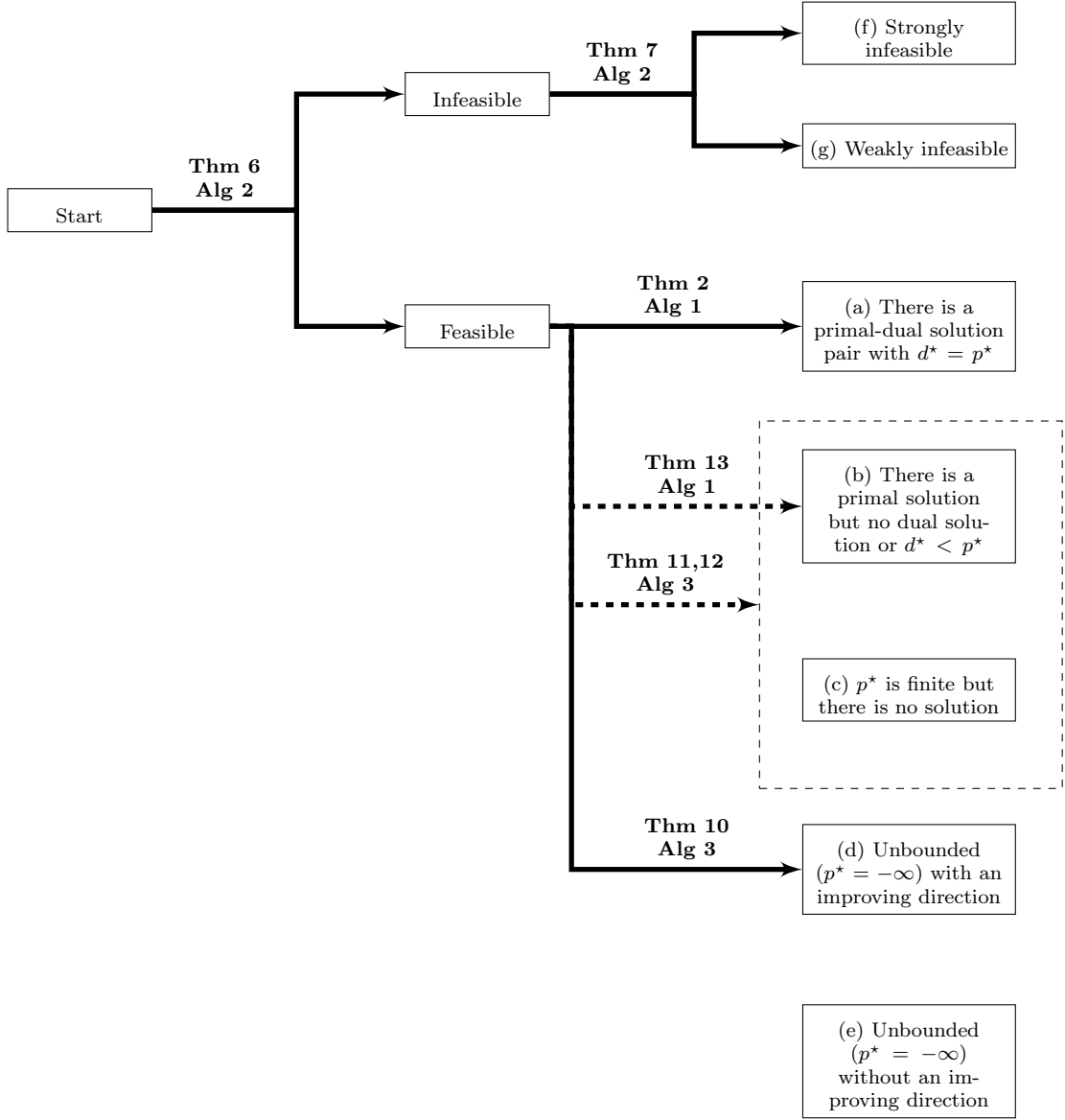

\subsection{Previous work}
Previously,
Bauschke, Combettes, Hare, Luke, and Moursi
have analyzed Douglas-Rachford splitting
in other pathological problems such as:
feasibility problems between 2 convex sets
\cite{BauschkeCombettesLuke2004_finding,bauschke2017}
feasibility problems between 2 convex sets
\cite{bauschke2016b},
and general setups
\cite{bauschke2012,bauschke2014,bauschke2016,Moursi2017_douglasrachford}.
Our work builds on these past results.

\section{Obtaining certificates from Douglas-Rachford Splitting/ADMM}
\label{s:theory}
The primal problem \eqref{P} is equivalent to
\begin{equation}
\begin{array}{ll}
\mbox{minimize}& f(x)+g(x),\\
\end{array}
\label{1}
\end{equation}
where
  \begin{align}
  f(x)&= c^{T}x+\delta_{\{x\,|\,Ax=b\}}(x)\nonumber\\
  g(x)&=\delta_{K}(x)\label{2},
  \end{align}
  and $\delta_C(x)$ is the indicator function of a set $C$ defined as
\[
\delta_C(x)=
\begin{cases}
0 \,\,\,\,\,\text{if}\,\,\, x\in C\\
\infty \,\,\text{if}\,\,\,x\notin C.
\end{cases}
\]

Douglas-Rachford splitting (DRS) \cite{LionsMercier1979_splitting} applied to \eqref{1} is
\begin{align}
x^{k+1/2}&=\Prox_{\gamma g}(z^k)\nonumber\\
x^{k+1}&=\Prox_{\gamma f}(2x^{k+1/2}-z^k)\label{eq:drs}\\
z^{k+1}&=z^k+x^{k+1}-x^{k+1/2},\nonumber
\end{align}
which updates $z^k$ to $z^{k+1}$ for $k=0,1,...$.
Given $\gamma>0$ and function $h$,
\[
\Prox_{\gamma h}(x) = \argmin_{z\in \mathbb{R}^n}\left\{
 h(z)+(1/2\gamma)\|z-x\|^2
\right\}
\]
denotes the proximal operator with respect to $\gamma h$.

\begin{proposition}
The DRS iteration \eqref{eq:drs} can be simplified to  
\begin{equation}
z^{k+1}=\tilde{T}(z^{k})+x_0-\gamma Dc,
\label{eq:drs3}
\end{equation}
which is also $z^{k+1}=T_1(z^k)$ with $T_1$ definied in \eqref{Operators}.
\end{proposition}
\begin{proof}
Given a nonempty closed convex set $C\subseteq \mathbb{R}^n$,
define the projection with respect to $C$ as
\[
P_C(x)=\argmin_{y\in C}\|y-x\|^2
\]
and the reflection with respect to $C$ as
\[
R_{C}(x)=2P_{C}(x)-x.
\]
Write $I$ to denote both the $n\times n$ identity matrix and the identity map
from $\mathbb{R}^n\rightarrow \mathbb{R}^n$.
Write $\mathbf{0}$ to denote the origin point in $\mathbb{R}^n$.
Define 
\begin{align}
    D&=I-A^T(AA^T)^{-1}A\nonumber\\
    x_0&=A^T(AA^T)^{-1}b=P_{\{x\,|\,Ax=b\}}(\mathbf{0}).\label{defx0}
\end{align}
Write $\mathcal{N}(A)$ for the null space of $A$ and $\mathcal{R}(A^T)$ for the 
range of $A^T$.
Then 
\begin{align*}
P_{\{x\,|\,Ax=b\}}(x)&=Dx+x_0,\\
P_{\mathcal{N}(A)}(x)&=Dx.
\end{align*}
Finally, define
\[
\tilde{T}(z)=\frac{1}{2}(I+R_{\mathcal{N}(A)} R_{K})(z).
\]

Now we can rewrite the DRS iteration \eqref{eq:drs} as
\begin{align}
x^{k+1/2}&=P_K(z^k)\nonumber\\
x^{k+1}&=D(2x^{k+1/2}-z^k)+x_0-\gamma Dc\label{eq:drs2}\\
z^{k+1}&=z^k+x^{k+1}-x^{k+1/2},\nonumber
\end{align}
which is equivalent to \eqref{eq:drs3}. \qed
\end{proof}

\paragraph{Relationship to ADMM.}
When we define $\nu^k=(1/\gamma)(z^k-x^k)$ and $\alpha=1/\gamma$,
reorganize, and reorder the iteration, 
the DRS iteration \eqref{eq:drs} becomes
\begin{align*}
x^{k}&=
\argmin_x
\left\{f(x)+x^T\nu^k+\frac{\alpha }{2}\|x-x^{k-1/2}\|^2\right\}
\\
x^{k+1/2}&=
\argmin_x
\left\{g(x)-x^T\nu^k+\frac{\alpha }{2}\|x-x^{k}\|^2\right\}\\
\nu^{k+1}&=\nu^k+\alpha(x^k-x^{k+1/2}),
\end{align*}
which is the alternating direction method of multipliers (ADMM).
In a certain sense, DRS and ADMM are equivalent \cite{Eckstein1989_splitting,EcksteinFukushima1994_reformulations,YanYin2016_self},
and we can equivalently say that the method of this paper
is based on ADMM.

\paragraph{Remark.}
Instead of \eqref{2}, we could have considered the more general form
  \begin{align*}
  f(x)&= (1-\alpha) c^{T}x+\delta_{\{x\,|\,Ax=b\}}(x), \\
  g(x)&= \alpha c^{T}x+\delta_{K}(x)
  \end{align*}
  with  $\alpha\in \mathbb{R}$.
By simplifying the resulting DRS iteration,
one can verify that the iterates are equivalent to the $\alpha=0$ case.
Since the choice of $\alpha$ does not affect the DRS iteration at all,
we will only work with the case $\alpha=0$.

\subsection{Convergence of DRS}
\label{ss:drsconv}
The subdifferential of a function $h:\mathbb{R}^n\rightarrow\mathbb{R}\cup \{\infty\}$ at $x$ is defined as
\[
\partial h(x)=\{u\in \mathbb{R}^n|\,\,h(z)\geq h(x)+u^T(z-x), \forall z\in \mathbb{R}^n\}.
\]

A point $x^{\star}\in\mathbb{R}^n$ is a solution of \eqref{1} if and only if 
\[
\mathbf{0}\in\partial(f+g)(x^{\star}).
\]
DRS, however, converges if and only if 
there is a point $x^{\star}$ such that 
\[
\mathbf{0}\in \partial f(x^{\star})+\partial g(x^{\star})
\]
(since $f$ and $g$ are closed convex proper functions).
In general,
\[
\partial f(x)+\partial g(x)
\subseteq \partial(f+g)(x)
\]
for all $x\in \mathbb{R}^n$,
but the two are not necessarily equal.

For example, consider the functions on $\mathbb{R}^2$
\[
f(x,y)=\left\{
\begin{array}{ll}
y&\text{if }x^2+y^2\le 1\\
\infty&\text{otherwise}
\end{array}
\right.
\qquad
g(x,y)=\left\{
\begin{array}{ll}
0&\text{if }x=1\\
\infty&\text{otherwise.}
\end{array}
\right.
\]
Then $f(x,y)+g(x,y)<\infty $ only at $(x,y)=(1,0)$, and therefore $(1,0)$ minimizes $f+g$. However, 
\[
\partial f(x,y)+\partial g(x,y)=
\left\{
\begin{array}{ll}
\{(a,1)\,|\,a\in \mathbb{R}\}&\text{if }(x,y)=(1,0)\\
\emptyset&\text{otherwise}
\end{array}
\right.
\]
whereas
\[
\partial (f+g)(x,y)=
\left\{
\begin{array}{ll}
\{(a,b)\,|\,a,b\in \mathbb{R}\}&\text{if }(x,y)=(1,0)\\
\emptyset&\text{otherwise.}
\end{array}
\right.
\]

We summarize the convergence of DRS in the theorem below. Its main part is a direct result of Theorem 1 of \cite{Rockafellar1976_monotone} and
Propositions~4.4 and 4.8 of \cite{Eckstein1989_splitting}. The convergence of $x^{k+1/2}$ and $x^{k+1}$ is due to \cite{Svaiter2011_WeakConvergence}. Therefore, we do not prove it.
\begin{theorem}
\label{thm:drsconv}
Consider the iteration \eqref{eq:drs3} with any starting point $z^0$.
If there is an $x$ such that
\[
\mathbf{0}\in \partial f(x)+\partial g(x),
\]
then $z^k$ converges to a limit $z^\star$,
$x^{k+1/2}\rightarrow x^\star=\Prox_{\gamma g}(z^\star)$,
$x^{k+1}\rightarrow x^\star=\Prox_{\gamma g}(z^\star)$,
and 
\[
\mathbf{0}\in \partial f(x^{\star})+\partial g(x^{\star}).
\]
If there is no $x$ such that
\[
\mathbf{0}\in \partial f(x)+\partial g(x),
\]
then $z^k$ diverges in that $\|z^k\|\rightarrow\infty$.
\end{theorem}

DRS can fail to find a solution to \eqref{P} even when one exists.
Slater's constraint qualification
is a sufficient condition that prevents such pathologies:
if \eqref{P} is strongly feasible,
then 
\[
\mathbf{0}\in \partial f(x^\star)+\partial g(x^\star)
\]
for all solutions $x^\star$ \cite[Theorem~23.8]{Rockafellar1970_convex}.
This fact and Theorem~\ref{thm:drsconv}
tell us that 
under Slater's constraint qualifications
DRS finds a solution of \eqref{P} if one exists.

The following theorem, however, provides a stronger, necessary and sufficient characterization of when
the DRS iteration converges. 
\begin{theorem}
\label{thm:0sub}
There is an $x^\star$ such that 
\[
\mathbf{0}\in \partial f(x^\star)+\partial g(x^\star)
\]
if and only if 
$x^\star$ is a solution to \eqref{P}, \eqref{D} has a solution, and $d^\star=p^\star$.
\end{theorem}

 Based on Theorem~\ref{thm:drsconv} and \ref{thm:0sub}
 we can determine whether we have case (a)
 with the iteration \eqref{eq:drs3}

 with any starting point $z^0$ and $\gamma>0$.
 \begin{itemize}
 \item
 If $\lim_{k\rightarrow\infty}\|z^k\|<\infty$, we have case (a), and vice versa.
 \item
 If $\lim_{k\rightarrow\infty}\|z^k\|=\infty$, we do not have case (a), and vice versa.
 \end{itemize}
With a finite number of iterations,
we test $\|z^k\|\ge M$ for some large $M>0$.
However, distinguishing the two cases can be numerically difficult
as the rate of $\|z^k\|\rightarrow \infty$ can be very slow.

\begin{proof}[Proof of Theorem~\ref{thm:0sub}]

This result follows from the exposition of  \cite{Rockafellar1974_conjugate}.
but we provide a proof that matches our notation.
 
The Lagrangian of \eqref{P} is
\[
\mathcal{L}(x,y,s)=c^Tx+y ^T(b-Ax)-s^Tx-\delta_{K^*}(s).
\]

We say $(x^\star,y^\star,s^\star)\in \mathbb{R}^n\times \mathbb{R}^m\times\mathbb{R}^n$ is a saddle point of $\mathcal{L}$ if
\begin{align*}
 x^\star&\in\argmin_{x\in \mathbb{R}^n}\mathcal{L}(x,y^\star,s^\star)\\
 (y^\star,s^\star)&\in\argmax_{y\in \mathbb{R}^m,s\in \mathbb{R}^n}\mathcal{L}(x^\star,y,s).
\end{align*}
It is well known that  $(x^\star,y^\star,s^\star)$ is a saddle point of $\mathcal{L}$ if and only if $x^\star$ is a solution to \eqref{P}, $(y^\star,s^\star)$ is a solution to \eqref{D}, and $p^\star=d^\star$ \cite{Rockafellar1974_conjugate}.

Now assume there is a saddle point $(x^\star,y^\star,s^\star)$.
Since $x^\star$ minimizes
$\mathcal{L}(x,y^\star,s^\star)$,
we have
$A^Ty^\star+s^\star-c=0$.
If $A^Ty^\star+s^\star-c\ne 0$,
then the terms of 
$\mathcal{L}(x,y^\star,s^\star)$
that depend on $x$ would be
$\nu^Tx$ for some $\nu \ne 0$.
This allows us to drive the value of 
$\mathcal{L}(x,y^\star,s^\star)$ to $-\infty$,
and there would be no minimizing $x^\star$.
By this same argument,
that
$y^\star$ maximizes
$\mathcal{L}(x^\star,y,s^\star)$
tells us
$Ax^\star=b$.

Since $s^\star$ maximizes
$\mathcal{L}(x^\star,y^\star,s)$,
we have
$x^\star\in K^{**}=K$
and $(x^\star)^Ts^\star=0$. To see why, note that the only terms
in $\mathcal{L}(x^\star,y^\star,s)$ that depend on $s$
are
\[
-(s^Tx^\star+\delta_{K^*}(s))
\]
If $x^\star\notin K^{**}=K$, then, by definition of dual
cones, there is a $s\in K^*$ such that $s^Tx^\star<0$.
By positively scaling this $s$, we can drive the value of $\mathcal{L}(x^\star,y^\star,s)$
to $\infty$, and there would be no maximizing $s^\star$.
If $x^\star\in K$, then 
\[
-(s^Tx^\star+\delta_{K^*}(s))\le 0,
\]
and the maximum is attained by $s=\mathbf{0}$.
So any $s^\star$ must satisfy $(x^\star)^Ts^\star=0$
to maximize $\mathcal{L}(x^\star,y^\star,s)$.

The other direction follows from taking the argument in the
other way.
\end{proof}

\subsection{Fixed-point iterations without fixed points}
\label{ss:fpe}
We say an operator $T:\mathbb{R}^n\rightarrow \mathbb{R}^n$ is nonexpansive if
\[
\|T(x)-T(y)\|^2\le \|x-y\|^2
\]
for all $x,y\in \mathbb{R}^n$. We say $T$ is firmly nonexpansive (FNE) if
\[
\|T(x)-T(y)\|^2\le \|x-y\|^2-\|(I-T)(x)-(I-T)(y)\|^2
\]
for all $x,y\in \mathbb{R}^n$.
(FNE operators are nonexpansive.)
In particular, all three operators defined in \eqref{Operators}
are FNE.
It is well known \cite{DavisYin2016_convergence} that
if a FNE operator $T$ has a fixed point,
its fixed-point iteration $z^{k+1}=T(z^k)$ converges to one
with rate
\[
\|z^{k}-z^{k+1}\|= o(1/\sqrt{k+1}).
\]

Now consider the case where a FNE operator $T$ has no fixed point,
which has been studied to a lesser extent.
In this case, the fixed-point iteration $z^{k+1}=T(z^k)$ diverges in  that $\|z^k\|\rightarrow\infty$ 
\cite[Theorem 1]{Rockafellar1976_monotone}.
Precisely in what manner $z^k$ diverges is characterized by the
\emph{infimal displacement vector} \cite{Pazy1971_asymptotic}.
Given a FNE operator $T$, we call
\[
v=P_{\overline{\ran (I-T)}}(\mathbf{0})
\]
the \emph{infimal displacement vector} of $T$.
To clarify, 
$\overline{\ran (I-T)}$ denotes the closure of the set
\[
\ran (I-T)
=\{x-T(x)\,|\,x\in \mathbb{R}^n\}.
\]
Because $T$ is FNE, the closed set
$\overline{\ran (I-T)}$
is convex \cite{Pazy1971_asymptotic},
so $v$ is uniquely defined.
We can interpret the infimal displacement vector $v$ as the asymptotic output of $I-T$
corresponding to the best effort to find a fixed point.

\begin{lemma}[Corollary 2.3 of \cite{BaillonBruckReich1978_asymptotic}]
\label{lem:infd}
Let $T$ be FNE,
and consider its fixed-point iteration
$z^{k+1}=T(z^k)$
with any starting point $z^0$.
Then
\[
z^{k}-z^{k+1}\rightarrow v=P_{\overline{\ran ({I-T}})}(\mathbf{0}).
\]
\end{lemma}
In  \cite{BaillonBruckReich1978_asymptotic}, Lemma~\ref{lem:infd} is proved in generality for 
nonexpansive operators, but we provide a simpler proof in our setting in Theorem~\ref{thm:fperate}.

When $T$ has a fixed point then $v=\mathbf{0}$, but $v=\mathbf{0}$ is possible even when $T$ has no fixed point.
In the following sections, we use 
Lemma~\ref{lem:infd}
to determine the status of a conic program,
but, in general,
$z^{k}-z^{k+1}\rightarrow v$ has no rate.
However, we only need to determine 
whether 
$\lim_{k\rightarrow\infty}(z^{k+1}-z^k)=\mathbf{0}$ or $\lim_{k\rightarrow\infty}(z^{k+1}-z^k)\ne \mathbf{0}$,
and we do so by checking whether $\|z^{k+1}-z^k\|\ge \varepsilon$ for some tolerance
$\varepsilon>0$.
For this purpose, the following rate of approximate convergence is good enough.
\begin{theorem}
\label{thm:fperate}
Let $T$ be FNE, and consider its fixed point iteration
\[
z^{k+1}=T(z^{k}),
\]
with any starting point $z^0$, then
\[
z^{k}-z^{k+1}\rightarrow v.
\]
And for any $\varepsilon>0$, there is an
$M_\varepsilon>0$ (which depends on $T$, $z^0$, and $\varepsilon$) such that
\[\|v\|\leq\min_{0\leq j\leq k}\|z^j-z^{j+1}\|\leq\|v\|+\frac{M_\varepsilon}{\sqrt{k+1}}+\frac{\varepsilon}{2}.
\]
\end{theorem}

\begin{proof}[Proof of Theorem~\ref{thm:fperate}]
For simplicity, we prove the result for $0<\varepsilon\le 1$,
although the Theorem~\ref{thm:fperate}
is true for $\varepsilon>1$ as well.

Given any $x_\varepsilon$, 
we use the triangle inequality to get
 \begin{align}
 \|z^k-z^{k+1}-v\|&=\|{T}^k(z^0)-T^{k+1}(z^0)-v\|\\
 &\leq \|({T}^k(z^0)-T^{k+1}(z^0))-({T}^k(x_{\varepsilon})-T^{k+1}(x_{\varepsilon}))\|+\|{T}^k(x_{\varepsilon})-T^{k+1}(x_{\varepsilon})-v\|.
 \label{eq:tri}
 \end{align}

To bound the second term,
pick an $x_{\varepsilon}$ such that
 \[
 \|x_{\varepsilon}-T(x_{\varepsilon})-v\|\leq\frac{\varepsilon^2}{4(2\|v\|+1)},
 \]
 which we can do since $v=P_{\overline{\ran (I-{T})}}(\mathbf{0})\in \overline{\ran (I-{T})}$.
 Since $T$ is nonexpansive, we get
 \[
 0\leq\|{T}^k(x_{\varepsilon})-{T}^{k+1}(x_{\varepsilon})\|-\|v\|\leq \frac{\epsilon^2}{4(2\|v\|+1)}.
 \]
 Since $v=P_{\overline{\ran (I-{T})}}(\mathbf{0})$,
 \[
 \|v\|^2\le y^Tv
 \]
 for any $y\in \overline{\ran (I-{T})}$.
 Putting these together we get
 \begin{equation}
 \begin{split}
 \|{T}^k(x_{\varepsilon})-{T}^{k+1}(x_{\varepsilon})-v\|^2 &=\|{T}^k(x_{\varepsilon})-{T}^{k+1}(x_{\varepsilon})\|^2+\|v\|^2-2 ({T}^k(x_{\varepsilon})-\tilde{T}^{k+1}(x_{\varepsilon}))^Tv \\
 & \leq \|{T}^k(x_{\varepsilon})-{T}^{k+1}(x_{\varepsilon})\|^2+\|v\|^2-2\|v\|^2\\
 &= (\|{T}^k(x_{\varepsilon})-{T}^{k+1}(x_{\varepsilon})\|+\|v\|)(\|{T}^k(x_{\varepsilon})-{T}^{k+1}(x_{\varepsilon})\|-\|v\|)\\
 &\leq (2\|v\|+\frac{\varepsilon^2}{4(2\|v\|+1)})\frac{\varepsilon^2}{4(2\|v\|+1)}\\
 &\leq (2\|v\|+1)\frac{\varepsilon^2}{4(2\|v\|+1)}=\frac{\varepsilon^2}{4}
 \end{split}
 \label{eq:sequence}
 \end{equation}
 for $0<\varepsilon\le 1$.

 Now let us bound the first term $\|({T}^k(z^0)-T^{k+1}(z^0))-({T}^k(x_\varepsilon)-T^{k+1}(x_\varepsilon))\|$ on the righthand side of \eqref{eq:tri}. Since $T$ is FNE, we have
\[
\|({T}^k(z^0)-T^{k+1}(z^0))-({T}^k(x_\varepsilon)-T^{k+1}(x_\varepsilon))\|^2 = \|T^k(z^0)-T^k(x_\varepsilon)\|^2-\|T^{k+1}(z^0)-T^{k+1}(x_\varepsilon)\|^2.
\]
Summing this inequality we have
\begin{equation}
\sum^k_{j=0}
\|({T}^k(z^0)-T^{k+1}(z^0))-({T}^k(x_\varepsilon)-T^{k+1}(x_\varepsilon))\|^2\le \|z^0-x_\varepsilon\|^2.
\label{eq:summable}
\end{equation}
\eqref{eq:tri}, \eqref{eq:sequence}, and \eqref{eq:summable} imply that
\[
z^k-z^{k+1}\rightarrow v.
\]
Furthermore,
\[
\min_{0\leq j \leq k}\|z^j-z^{j+1}-v\|\leq \frac{M_{\varepsilon}}{\sqrt{k+1}}+\frac{\varepsilon}{2},
\]
where $M_{\varepsilon}=\|z^0-x_{\varepsilon}\|$. As a result,
\[\|v\|\leq\min_{0\leq j\leq k}\|z^j-z^{j+1}\|\leq\|v\|+\frac{M_\varepsilon}{\sqrt{k+1}}+\frac{\varepsilon}{2}.
\]
\end{proof}

\subsection{Feasibility and infeasibility}
We now return to the specific conic programs.
Consider the operator $T_2$ defined by
$T_2(z)=\tilde{T}(z)+x_0$.
As mentioned, we can view $T_2$ as the DRS operator 
with $c$ set to $\mathbf{0}$ in \eqref{P}.

The infimal displacement vector of $T_2$ has a nice geometric interpretation:
it is the best approximation displacement
between the sets $K$ and $\{x\,|\,Ax=b\}$,
and $\|v\|=d(K,\{x\,|\,Ax=b\})$.
\begin{theorem}[Theorem~3.4 of \cite{BauschkeCombettesLuke2004_finding}, Proposition 11.22 of \cite{Moursi2017_douglasrachford}]
\label{thm:bestapprox}
The operator $T_2$ defined by
$T_2(z)=\tilde{T}(z)+x_0$, where $x_0$ is given in \eqref{defx0},
has the infimal displacement vector
$v=P_{\overline{K-\{x\,|\,Ax=b\}}}(\mathbf{0})$.
\end{theorem}
We can further understand $v$ in terms of the projection
$P_{\overline{P_{\mathcal{R}(A^T)}(K)}}$.
Note that $P_{\mathcal{R}(A^T)}(K)$ is a cone because $K$ is. $P_{\mathcal{R}(A^T)}(K)$
is not always closed, but its closure $\overline{P_{\mathcal{R}(A^T)}(K)}$ is. 
\begin{lemma}[Interpretation of $v$] The infimal displacement vector $v$ of $T_2$ satisfies
\label{lem:v_interpret}
\[
v=P_{\overline{K-\{x\,|\,Ax=b\}}}(\mathbf{0})
=P_{\overline{P_{\mathcal{R}(A^T)}(K)}-x_0}(\mathbf{0})=P_{\overline{P_{\mathcal{R}(A^T)}(K)}}(x_0)-x_0,
\]
where $x_0$ is given in \eqref{defx0} and $K$ is any nonempty set.
\end{lemma}

Combining the discussion of Section~\ref{ss:fpe}
with Theorem~\ref{thm:bestapprox}
gives us Theorems \ref{thm:feas} and \ref{thm:sfeas}.
\begin{theorem}[Certificate of feasibility]
\label{thm:feas}
Consider the iteration $z^{k+1}=T_2(z^k)$
with any starting point $z^0\in \mathbb{R}^n$, then \begin{enumerate}
    \item \eqref{P} is feasible if and only if $z^k$ converges, in this case $x^{k+1/2}$ converges to a feasible point of \eqref{P}.
    \item \eqref{P} is infeasible if and only if  $z^k$ diverges in that $\|z^k\|\rightarrow\infty$.
\end{enumerate}
\end{theorem}

\begin{theorem}[Certificate of strong infeasibility]
\label{thm:sfeas}
Consider the iteration $z^{k+1}=T_2(z^k)$
with any starting point $z^0$, we have $z^{k}-z^{k+1}\rightarrow v$ and 
\begin{enumerate}
    \item \eqref{P} is strongly infeasible if and only if $v\neq\mathbf{0}$.
    \item \eqref{P} is weakly infeasible or feasible if and only if $v=\mathbf{0}$.
\end{enumerate}
\end{theorem}

When \eqref{P} is strongly infeasible,  we can obtain a separating hyperplane from $v$.
\begin{theorem}[Separating hyperplane]
\label{thm:hyperplane}
Consider the iteration $z^{k+1}=T_2(z^k)$
with any starting point $z^0$, we have $z^{k}-z^{k+1}\rightarrow v$, 
\eqref{P} is strongly infeasible if and only if $v\ne \mathbf{0}$,
and the hyperplane
\[
\{x\,|\,h^Tx=\beta\},
\]
where $h=-v\in K^*\cap \mathcal{R}(A^T)$ and $\beta =-(v^Tx_0)/2>0$,
strictly separates $K$ and $\{x\,|\,Ax=b\}$.
More precisely,
for any $y_1\in K$ and $y_2\in \{x\,|\,Ax=b\}$
we have
\[
h^Ty_1<\beta <h^Ty_2.
\]

\end{theorem}

Based on Theorems \ref{thm:feas}, \ref{thm:sfeas}, and
\ref{thm:hyperplane},
we can determine feasibility, weak infeasiblity, and strong infeasibility and obtain a strictly separating hyperplane
if one exists
with the iteration $z^{k+1}=T_2(z^k)$
with any starting point $z^0$.
\begin{itemize}
\item
$\lim_{k\rightarrow\infty}\|z^k\|<\infty$ if and only if
\eqref{P} is feasible.
\item
$\lim_{k\rightarrow\infty}\|z^{k}-z^{k+1}\|>0$ if and only if
\eqref{P} is strongly infeasible, and Theorem~\ref{thm:hyperplane}
provides a strictly separating hyperplane.
\item
$\lim_{k\rightarrow\infty}\|z^k\|=\infty$ and
$\lim_{k\rightarrow\infty}\|z^{k}-z^{k+1}\|=0$ if and only if
\eqref{P} is weakly infeasible.
\end{itemize}
With a finite number of iterations,
we distinguish the three cases by testing
$\|z^{k+1}-z^k\|\le \varepsilon$ and $\|z^{k}\|\ge M$
for some small $\varepsilon>0$ and large $M>0$.
By Theorem~\ref{thm:fperate}, 
we can distinguish strong infeasibility
from weak infeasibility or feasibility at a rate of $O(1/\sqrt{k})$.
However, distinguishing feasibility from weak infeasibility can be numerically difficult
as the rate of $\|z^k\|\rightarrow\infty$ can be very slow when \eqref{P} is weakly infeasible.

\begin{proof}[Proof of Lemma~\ref{lem:v_interpret}]
Remember that
by definition \eqref{defx0},
we have $x_0\in \mathcal{R}(A^T) $ and
\[
\{x\,|\,Ax=b\}=x_0+\mathcal{N}(A)=x_0-\mathcal{N}(A).
\]
Also note that for any $y\in \mathbb{R}^n$, we have
\[
y+\mathcal{N}(A)=P_{\mathcal{R}(A^T)}(y)+\mathcal{N}(A).
\]
So
\[
K-\{x\,|\,Ax=b\}
=K+\mathcal{N}(A)-x_0=
P_{\mathcal{R}(A^T)}(K)-x_0+ \mathcal{N}(A),
\]
and
\begin{equation}
\overline{K-\{x\,|\,Ax=b\}}
=
\overline{P_{\mathcal{R}(A^T)}(K)+ \mathcal{N}(A)}-x_0
=
\overline{P_{\mathcal{R}(A^T)}(K)}-x_0+ \mathcal{N}(A).
\end{equation}
Since $x_0\in \mathcal{R}(A^T)$,
we have
$\overline{P_{\mathcal{R}(A^T)}(K)}-x_0\subseteq \mathcal{R}(A^T)$,
and, 
in particular, $\overline{P_{\mathcal{R}(A^T)}(K)}-x_0$ is orthogonal to the subspace $\mathcal{N}(A)$.
Recall
\[
v=P_{\overline{P_{\mathcal{R}(A^T)}(K)}-x_0+\mathcal{N}(A)}(\mathbf{0}).
\]
So 
$v\in \overline{P_{\mathcal{R}(A^T)}(K)}-x_0\subseteq \mathcal{R}(A^T)$
and
\[
v=P_{\overline{P_{\mathcal{R}(A^T)}(K)}-x_0}(\mathbf{0}).
\]Finally,
\[
v=
\argmin_{x\in \overline{P_{\mathcal{R}(A^T)}(K)}-x_0}
\left\{\|x\|_2^2\right\}
=
\argmin_{y\in \overline{P_{\mathcal{R}(A^T)}(K)}}
\left\{\|y-x_0\|_2^2\right\}-x_0=P_{\overline{P_{\mathcal{R}(A^T)}(K)}}(x_0)-x_0
\]
\end{proof}

\begin{proof}[Proof of Theorem~\ref{thm:hyperplane}]
Note that
\begin{align*}
v&=P_{\overline{K-\{x\,|\,Ax=b\}}}(\mathbf{0})
=P_{\overline{K+\mathcal{N}(A)-x_0}}(\mathbf{0})=P_{\overline{K+\mathcal{N}(A)}}(x_0)-x_0
\end{align*}
Using 
$I=P_{K^*\cap \mathcal{R}(A^T)}+P_{-(K^*\cap \mathcal{R}(A^T))^*}$ and $(K^*\cap \mathcal{R}(A^T))^*=\overline{K+\mathcal{N}(A)}$ \cite{BauschkeCombettes2011_convex}, we have
\[
v=P_{\overline{K+\mathcal{N}(A)}}(x_0)-x_0=-P_{-(K^*\cap \mathcal{R}(A^T))}(x_0)=P_{K^*\cap \mathcal{R}(A^T)}(-x_0).
\]
Since the projection operator is FNE, we have 
\[
-v^Tx_0=(v-\mathbf{0})^T(-x_0-\mathbf{0})\ge 
\|
P_{K^*\cap\mathcal{R}(A^T)}(-x_0)
\|^2=\|v\|^2>0
\]
and therefore $v^Tx_0<0, \beta=-v^Tx_0/2>0.$

So for any $y_1\in K$ and $y_2\in \{x\,|\,Ax=b\}$, we have
\[
h^Ty_1=-v^Ty_1\le 0<-(v^Tx_0)/2=\beta <-v^Tx_0 =h^Ty_2,
\]
where we have used $h=-v=-P_{K^*\cap \mathcal{R}(A^T)}(-x_0)\in -K^*$ in the first inequality.
\end{proof}

\subsection{Modifying affine constraints to achieve strong feasibility}
Strongly feasible problems are, loosely speaking, the good cases
that are easier to solve, compared
to weakly feasible or infeasible problems. Given a problem that is not strongly feasible,
how to minimally modify the problem to achieve strong feasibility
is often useful to know.

The limit $z^k-z^{k+1}\rightarrow v$ informs us of how to do this.
When $d(K,\{x\,|\,Ax=b\})=\|v\|>0$, the constraint
$K\cap \{x\,|\,A(x-y)=b\}$ is infeasible for any $y$ such that $\|y\|<\|v\|$.
In general, the constraint $K\cap \{x\,|\,A(x-v)=b\}$ can be feasible or weakly infeasible,
but is not strongly feasible.
The constraint $K\cap \{x\,|\,A(x-v-d)=b\}$
is strongly feasible for an arbitrarily small $d\in\relint K$.
In other words, 
$K\cap \{x\,|\,A(x-v-d)=b\}$
achieves strong feasibility
with the minimal modification (measured by the Euclidean norm $\|\cdot\|$)
to the original constraint $K\cap \{x\,|\,Ax=b\}$.

\begin{theorem}[Achieving strong feasibility]
\label{thm:achieve_feas}
Let 
$v=P_{\overline{K-\{x\,|\,Ax=b\}}}(\mathbf{0})$,
and let $d$ be any vector satisfying $d\in \relint K$.
Then the constraint
$K\cap \{x\,|\,A(x-v-d)=b\}$
is strongly feasible, i.e.,
there is an $x$ such that
$x\in
\relint K\cap \{x\,|\,A(x-v-d)=b\}$.
\end{theorem}
 \begin{proof}[Proof of Theorem~\ref{thm:achieve_feas}]

By Lemma~\ref{lem:v_interpret}
we have
\begin{align}\label{eq:vx0}
 v+x_0\in \overline{P_{\mathcal{R}(A^T)}(K)}.
 \end{align}
Because $P_{\mathcal{R}(A^T)}$ is a linear transformation,
by Lemma~\ref{lem:linrelint} below
\[
P_{\mathcal{R}(A^T)}(\relint K)
=\relint 
P_{\mathcal{R}(A^T)}(K).
\]
Since $d\in \relint K$,
\begin{align}\label{eq:Pd}
P_{\mathcal{R}(A^T)}(d)\in 
P_{\mathcal{R}(A^T)}(\relint K)
=\relint 
P_{\mathcal{R}(A^T)}(K).
\end{align}
Applying Lemma~\ref{lem:coneint} to \eqref{eq:vx0} and \eqref{eq:Pd}, we have
\[
v+x_0+
P_{\mathcal{R}(A^T)}(d)\in
\relint {P_{\mathcal{R}(A^T)}(K)}
=
P_{\mathcal{R}(A^T)}(\relint K).
 \]
 Finally we have
\[
0\in  P_{\mathcal{R}(A^T)}(\relint K)-x_0-v-d+\mathcal{N}(A)=
\relint K-\{x\,|\,A(x-v-d)=b\}.
\]
\end{proof}

\begin{lemma}[Theorem~6.6 of \cite{Rockafellar1970_convex}]
\label{lem:linrelint}
If $A(\cdot)$ is a linear transformation and $C$ is a 
convex set,
then $A(\relint C)=\relint A(C)$.
\end{lemma}

\begin{lemma}\label{lem:coneint}
Let $K$ be a convex cone.
If $x\in K$ and $y\in \relint K$, then $x+y\in \relint K$.
\end{lemma}
\begin{proof}
Since $K$ is a convex set and $y\in \relint K$, 
we have $(1/2)x+(1/2)y\in \relint K$.
Since $K$ is a cone,  $(1/2)(x+y)\in \relint K$ implies $x+y\in \relint K$.
\end{proof}

\subsection{Improving direction}
\label{ss:impdir}
\eqref{P} has an improving direction if and only if the dual problem \eqref{D} is strongly infeasible:
\[
0<
d(0,K^\star+\mathcal{R}(A^T)-c)=
d(\{(y,s)\,|\,A^Ty+s=c\},\{(y,s)\,|\,s\in K^*=c\}).
\]

\begin{theorem}[Certificate of improving direction]
\label{thm:impdir}
Exactly one of the following is true:
 \begin{enumerate}
 \item
 \eqref{P} has an improving direction, \eqref{D} is strongly infeasible, 
 and $P_{\mathcal{N}(A)\cap K}(-c)\ne \mathbf{0}$ is an improving direction.
 \item
  \eqref{P} has no improving direction, \eqref{D} is feasible or weakly infeasible,
  and $P_{\mathcal{N}(A)\cap K}(-c)=\mathbf{0}$.
 \end{enumerate} 
Furthermore,
\[
P_{\mathcal{N}(A)\cap K}(-c)=
P_{\overline{K^*+\mathcal{R}(A^T)-c}}(\mathbf{0}).
\]
\end{theorem}

\begin{theorem}
\label{thm:impdir2}
Consider the iteration $z^{k+1}=T_3(z^k)=\tilde{T}(z^k)-\gamma Dc$
with any starting point $z^0$ and $\gamma>0$.
If \eqref{P} has an improving direction, then 
\[
d=\lim_{k\rightarrow \infty}z^{k+1}-z^k=P_{\overline{K^*+\mathcal{R}(A^T)-c}}(\mathbf{0})\neq \mathbf{0}
\]
gives one. If \eqref{P} has no improving direction, then
\[
\lim_{k\rightarrow \infty}z^{k+1}-z^k=\mathbf{0}.
\]
\end{theorem}

 Based on Theorem~\ref{thm:impdir} and \ref{thm:impdir2}
 we can determine whether there is an improving direction
 and find one if one exists with the iteration
 $z^{k+1}=\tilde{T}(z^k)-\gamma Dc$
 with any starting point $z^0$ and $\gamma>0$.
 \begin{itemize}
 \item
 $\lim_{k\rightarrow\infty}z^{k+1}-z^k=\mathbf{0}$ if and only if there
 is no improving direction.
 \item
 $\lim_{k\rightarrow\infty}z^{k+1}-z^k=d\ne \mathbf{0}$ if and only if $d$ is an improving direction.
 \end{itemize}
With a finite number of iterations,
we test 
$\|z^{k+1}-z^k\|\le \varepsilon$
for some small $\varepsilon>0$.
By Theorem~\ref{thm:fperate},
we can distinguish whether there is an improving direction or not 
at a rate of $O(1/\sqrt{k})$.

We need the following theorem for Section~\ref{ss:other}, it is proved similarly to \ref{thm:feas} below.
\begin{theorem}
\label{thm:impdir3}
Consider the iteration
\[
z^{k+1}=\tilde{T}(z^k)-\gamma Dc
\]
with any starting point $z^0$ and $\gamma>0$.
If \eqref{D} is feasible, then
$z^k$ converges.
If \eqref{D} is infeasible, then 
$z^k$ diverges in that $\|z^k\|\rightarrow\infty$.
\end{theorem}

\begin{proof}[Proof of Theorem~\ref{thm:impdir}]
This result is known \cite{LuoSturmZhang1997_duality},
but we provide a proof that matches our notation.

\eqref{P} has no improving direction if and only if 
\[
\{x\in \mathbb{R}^n|x\in \mathcal{N}(A)\cap K, c^Tx<0\}=\emptyset,
\]
which is equivalent to
$c^Tx\ge 0$ for all $\in \mathcal{N}(A)\cap K$.
This is in turn equivalent to $c\in (\mathcal{N}(A)\cap K)^*$.
So
\[
-c=P_{-(\mathcal{N}(A)\cap K)^*}(-c).
\]
if and only if there is no improving direction,
which holds if and only if
\[
0=P_{\mathcal{N}(A)\cap K}(-c).
\]

Assume there is an improving direction.
Since the projection operator is firmly nonexpansive, we have
\[
0<
\|P_{\mathcal{N}(A)\cap K}(-c)\|^2\le 
(P_{\mathcal{N}(A)\cap K}(-c))^T(-c).
\]
This simplifies to
\[
(P_{\mathcal{N}(A)\cap K}(-c))^Tc<0,
\]
and we conclude $P_{\mathcal{N}(A)\cap K}(-c)$ is an improving direction.

Using the fact that $(\mathcal{N}(A)\cap K)^*=\overline{K^*+\mathcal{R}(A^T)}$, 
we have
\[
P_{\mathcal{N}(A)\cap K}(-c)=
-P_{\mathcal{N}(A)\cap K}(c)=
(P_{\overline{K^*+\mathcal{R}(A^T)}}-I)(c)=
P_{\overline{K^*+\mathcal{R}(A^T)-c}}(\mathbf{0}),
\]
where we have used the identity $I=P_{\mathcal{N}(A)\cap K}+P_{\overline{K^*+\mathcal{R}(A^T)}}$ in the second equality.
\end{proof}

\begin{proof}[Proof of Theorem~\ref{thm:impdir2} and \ref{thm:impdir3}]
Using the identities
$I=P_{\mathcal{N}(A)}+P_{\mathcal{R}(A^T)}$, $I=P_{K}+P_{-K^*}$,
and
$R_{\mathcal{R}(A^T)-\gamma c}(z)=R_{\mathcal{R}(A^T)}(z)-2\gamma Dc$,
we have
\[
T_3(z)=\tilde{T}(z)-\gamma Dc=
\frac{1}{2}(I+R_{\mathcal{R}(A^T)-\gamma c}R_{-K^*})(z).
\]
In other words, we can interpret the fixed point iteration
\[
z^{k+1}=\tilde{T}(z^k)-\gamma Dc
\]
as the DRS iteration on 
\[
\begin{array}{ll}
\mbox{minimize}&0\\
\mbox{subject to}& x\in \mathcal{R}(A^T)-\gamma c\\
&x\in -K^*.
\end{array}
\]
This proves Theorem~\ref{thm:impdir3}.

Using Lemma \ref{lem:infd}, 
applying  Theorem~3.4 of \cite{BauschkeCombettesLuke2004_finding} as we did for Theorem \ref{thm:bestapprox},
and applying Theorem~\ref{thm:impdir}, we get
\begin{align*}
z^k-z^{k+1} &\rightarrow P_{\overline{\ran ({I-T_3}})}(\mathbf{0})\\
&=P_{\overline{-K^*-\mathcal{R}(A^T)}+\gamma c}(\mathbf{0})\\
&=-\gamma P_{\overline{K^*+\mathcal{R}(A^T)-c}}(\mathbf{0})\\
&=-\gamma P_{\mathcal{N}(A)\cap K}(-c).
\end{align*}

\end{proof}

\subsection{Modifying the objective to achieve finite optimal value}

Similar to \ref{thm:achieve_feas}, we can achieve strong feasibility of \eqref{D} by modifying $c$, and \eqref{P} will have a finite optimal value.

\begin{theorem}[Achieving finite $p^{\star}$]
\label{thm:achieve_finite}
Let $w=P_{\overline{K^*+\mathcal{R}(A^T)-c}}(\mathbf{0})$, and let $s$ be any vector satisfying $s\in \relint{K^*}$. If \eqref{P} is feasible and has an unbounded direction, then by replacing $c$ with $c'=c+w+s$, \eqref{P} will have a finite optimal value. 

\end{theorem}
\begin{proof}[Proof of Theorem~\ref{thm:achieve_finite}]
Similar to Lemma~\ref{lem:v_interpret}, we have 
\[
w=P_{\overline{P_{\mathcal{N}(A)}(K^*)}-P_{\mathcal{N}(A)(c)}}(\mathbf{0}).
\]
And similar to Theorem~\ref{thm:achieve_feas}, the new constraint of \eqref{D} 
$$K^*\cap \{c+w+s-A^Ty\}$$ 
is strongly feasible. The constraint of \eqref{P} is still $K\cap\{x\,|\,Ax=b\}$, which is feasible. By weak duality of we conclude that the optimal value of \eqref{P} becomes finite.
\end{proof}
\subsection{Other cases}
\label{ss:other}
So far, we have discussed how to identify and certify cases (a), (d), (f), and (g).
We now discuss sufficient conditions 
to certify the remaining cases.

The following theorem follows from weak duality.
\begin{theorem}[\cite{Rockafellar1974_conjugate} Certificate of finite $p^{\star}$]
\label{thm:finitep}
If \eqref{P} and \eqref{D} are feasible, then $p^{\star}$ is finite.
\end{theorem}

Based on Theorem~\ref{thm:impdir3},
we can determine whether \eqref{D} is feasible with the iteration $z^{k+1}=T_3(z^k)=\tilde{T}(z^k)-\gamma Dc$,

with any starting point $z^0$ and $\gamma>0$.
 \begin{itemize}
 \item
 $\lim_{k\rightarrow\infty}\|z^{k}\|<\infty$ if and only if \eqref{D} is feasible.
 \item
 $\lim_{k\rightarrow\infty}\|z^{k}\|=\infty$ if and only if \eqref{D} is infeasible.
 \end{itemize}
With a finite number of iterations,
we test 
$\|z^k\|\ge M$
for some large $M>0$.
However, distinguishing the two cases can be numerically difficult
as the rate of $\|z^k\|\rightarrow \infty$ can be very slow.

\begin{theorem}[Primal iterate convergence]
\label{thm:pconv}
Consider the DRS iteration as defined in \eqref{eq:drs2}
with any starting point $z^0$.
Assume \eqref{P} is feasible, if $x^{k+1/2}\rightarrow x^\infty$ and $x^{k+1}\rightarrow x^\infty$,
then $x^\infty$ is primal optimal,
even if $z^k$ doesn't converge.
\end{theorem}

When running the fixed-point iteration with $T_1(z)=\tilde{T}(z)+x_0-\gamma Dc$,
if $\|z^k\|\rightarrow \infty$ but $x^{k+1/2}\rightarrow x^\infty$ and $x^{k+1}\rightarrow x^\infty$,
then we have case (b), but the converse is not necessarily true.

\paragraph{Examples for Theorem~\ref{thm:finitep}.}
Consider the following problem in case (c):
\begin{equation*}
\begin{array}{ll}
\mbox{minimize}& x_3\\
\mbox{subject to}& x_1=\sqrt{2} \\
&2x_2x_3\ge x_1^2.
\end{array}
\end{equation*}
Its dual problem is
\[
\begin{array}{ll}
\mbox{maximize}& \sqrt{2}y\\
\mbox{subject to} & y^2\le 1,
\end{array}
\]
which is feasible.
Based on diagnostics discussed in the previous sections
and the fact that the dual problem is feasible,
one can conclude that 
we have either case (b) or (c) but not case (e).

Consider the following problem in case (e):
\begin{equation*}
\begin{array}{ll}
\mbox{minimize}& x_1\\
\mbox{subject to}& x_2=1 \\
&2x_2x_3\ge x_1^2
\end{array}
\end{equation*}
Its dual problem is
\[
\begin{array}{ll}
\mbox{maximize}& y\\
\mbox{subject to} & 1\le 0,
\end{array}
\]
which is infeasible.
The diagnostics discussed in the previous sections
allows us to conclude that we have case
(b), (c), or (e).
The fact that the dual problem is infeasible
may suggest that we have case (e), there is no such guarantee.
Indeed, the dual must be infeasible 
if we have case (e),
but the converse is not necessarily true.

\paragraph{Example for Theorem~\ref{thm:pconv}}
Consider the following problem in  case (b):
\begin{equation*}
\begin{array}{ll}
\mbox{minimize}& x_2\\
\mbox{subject to}& x_1=x_3=1\\
&x_3\ge \sqrt{x_1^2+x_2^2}.
\end{array}
\end{equation*}
When we run the iteration \eqref{eq:drs2},
we can empirically observe
that $x^{k+1/2}\rightarrow x^\star$ and $x^{k+1}\rightarrow x^\star$,
and conclude that we have case (b).

Again, consider the following problem in case (e):
\begin{equation*}
\begin{array}{ll}
\mbox{minimize}& x_1\\
\mbox{subject to}& x_2=1 \\
&2x_2x_3\ge x_1^2
\end{array}
\end{equation*}
When we run the iteration \eqref{eq:drs2},
we can empirically observe
that $x^{k+1/2}$ and $x^{k+1}$ do not converge.
The diagnostics discussed in the previous sections
allows us to conclude that we have case
(b), (c), or (e).
The fact that $x^{k+1/2}$ and $x^{k+1}$ do not converge
may suggest that we have case (c) or (e), but there is no such guarantee.
Indeed, $x^{k+1/2}$ and $x^{k+1}$ must not converge
when we have case (c) or (e),
but the converse is not necessarily true.

\paragraph{Counterexample for Theorem~\ref{thm:finitep} and \ref{thm:pconv}}
The following example shows that
the converses of Theorem~\ref{thm:finitep} and \ref{thm:pconv} are not true.
Consider the following problem in case (b):
\begin{equation*}
\begin{array}{ll}
\mbox{minimize}& x_1\\
\mbox{subject to}& x_2-x_3=0 \\
\mbox{}& x_3\ge \sqrt{x_1^2+x_2^2},
\end{array}
\end{equation*}
which has the solution set $\{(0,t,t)\,|\,t\in R\}$
and optimal value $p^\star=0$.
Its dual problem is
\[
\begin{array}{ll}
\mbox{maximize} &0\\
\mbox{subject to} &y\ge \sqrt{y^2+1},
\end{array}
\]
which is infeasible.
This immediately tells us that
$p^\star>-\infty$ is possible even when $d^\star=-\infty$.

Furthermore, the $x^{k+1/2}$ and $x^{k+1}$ iterates do not converge even though there is a solution.
Given $z^0=(z^0_1,z^0_2,0)$, the iterates
$z^{k+1}=(z^{k+1}_1, z^{k+2}_2,z^{k+1}_3)$ are:
\begin{align*}
 z^{k+1}_1&=\frac{1}{2}z^k_1-\gamma \\
 z^{k+1}_2&=\frac{1}{2}z^k_2+\frac{1}{2}\sqrt{(z^k_1)^2+(z^k_2)^2}\\
 z^{k+1}_3&=0.
\end{align*}
So $x^{k+1/2}=P_{K}(z^k)$ satisfies $x^k_1\rightarrow -2\gamma, x^k_2\rightarrow \infty$ and $x^k_3\rightarrow \infty$,
and we can see that $x^{k+1/2}$ does not converge to the solution set.

\begin{proof}[Proof of Theorem~\ref{thm:pconv}]
Define 
\begin{align*}
x^{k+1/2}&=\Prox_{\gamma g}(z^k)\\
x^{k+1}&=\Prox_{\gamma f}(2x^{k+1/2}-z^k)\\
z^{k+1}&=z^k+x^{k+1}-x^{k+1/2}
\end{align*}
as in \eqref{eq:drs2}
Define
\begin{align*}
\tilde{\nabla}g(x^{k+1/2})&=(1/\gamma)(z^k-x^{k+1/2})\\
\tilde{\nabla}f(x^{k+1})&=(1/\gamma)(2x^{k+1/2}-z^k-x^{k+1}).
\end{align*}
It's simple to verify that
\begin{align*}
\tilde{\nabla}g(x^{k+1/2})&\in\partial g(x^{k+1/2})\\
\tilde{\nabla}f(x^{k+1})&\in\partial f(x^{k+1}).
\end{align*}

Clearly,
\[
\tilde{\nabla}g(x^{k+1/2})+\tilde{\nabla}f(x^{k+1})=(1/\gamma)(x^{k+1/2}-x^{k+1}).
\]
We also have
\[
z^{k+1}=z^k
-\gamma \tilde{\nabla}g(x^{k+1/2})-\gamma \tilde{\nabla}f(x^{k+1})=x^{k+1/2}-\gamma \tilde{\nabla}f(x^{k+1})
\]

Consider any $x\in K\cap \{x\,|\,Ax=b\}$. Then, by convexity of $f$ and $g$,
\begin{align*}
g(x^{k+1/2})-g(x)+f(x^{k+1})-f(x)
&\le 
\tilde{\nabla}g(x^{k+1/2})^T(x^{k+1/2}-x)
+
\tilde{\nabla}f(x^{k+1})^T(x^{k+1}-x)\\
&=
(\tilde{\nabla}g(x^{k+1/2})+\tilde{\nabla}f(x^{k+1}))^T(x^{k+1/2}-x)
+\tilde{\nabla}f(x^{k+1})^T(x^{k+1}-x^{k+1/2})\\
&=(x^{k+1}-x^{k+1/2})^T
(\tilde{\nabla}f(x^{k+1})-(1/\gamma)(x^{k+1/2}-x))\\
&=(1/\gamma)(x^{k+1}-x^{k+1/2})^T(x-z^{k+1})
\end{align*}
We take the liminf on both sides and use
Lemma~\ref{lem:delta} below to get
\[
g(x^{\infty})+f(x^{\infty})\le g(x)+f(x).
\]

Since this holds for any $x\in K\cap \{x\,|\,Ax=b\}$, $x^\infty$ is optimal.
\end{proof}

\begin{lemma}
\label{lem:delta}
Let $\Delta^1,\Delta^2,\dots$ be a sequence in $\mathbb{R}^n$.
Then
\[
\liminf_{k\rightarrow\infty} (\Delta^k)^T\sum^k_{i=1}(-\Delta^i)\le 0.
\]
\end{lemma}
\begin{proof}
Assume for contradiction that 
\[
\liminf_{k\rightarrow\infty} (\Delta^k)^T\sum^k_{i=1}(-\Delta^i)>2\varepsilon
\]
for some $\varepsilon>0$.
Since the initial part of the sequence is irrelevant,
assume without loss of generality
that
\[
(\Delta^j)^T\sum^j_{i=1}\Delta^i<-\varepsilon
\]
for $j=1,2,\dots$,
summing both sides gives us, for all $k=1,2,...$
\begin{align*}
\sum^k_{j=1}
(\Delta^j)^T\sum^j_{i=1}\Delta^i<-\varepsilon k.\\
\end{align*}
Define
\[ \mathbbm{1}\{i\le j\}=
\begin{cases}
1, \text{if}\,\, i\leq j,\\
0, \text{otherwise}.
\end{cases}
\]
We have
\begin{align*}
\sum^k_{j=1}
 \sum^k_{i=1}
 (\Delta^j)^T\Delta^i
 \mathbbm{1}\{i\le j\}
 <-\varepsilon k,\\
0\leq \frac{1}{2}
\left\|
\sum^k_{i=1}\Delta^i\right\|^2
+
\frac{1}{2}
\sum^k_{i=1}
\left\|\Delta^i\right\|^2
<-\varepsilon k,  
\end{align*}
which is a contradiction.
\end{proof}

\subsection{The algorithms}
\label{ss:alg}
In this section, we collect the discussed classification results
as thee algorithms.
The full algorithm is simply running
Algorithms~\ref{alg1}, \ref{alg2}, and \ref{alg3},
and applying flowchart of Figure~\ref{Fig.1}.

\begin{algorithm}
\caption{Finding a solution}
\label{alg1}
\begin{algorithmic}[0]
\State Parameters: $\gamma$, $M$, $\varepsilon$, $z^0$
\For{$k=1,\dots$}
\State $x^{k+1/2}=P_K(z^k)$
\State $x^{k+1}=D(2x^{k+1/2}-z^k)+x_0-\gamma Dc$
\State $z^{k+1}=z^k+x^{k+1}-x^{k+1/2}$
\EndFor
\If{$\|z^k\|< M$}
\State Case (a)
\State $x^{k+1/2}$ and $x^{k+1}$ solution
\ElsIf{$x^{k+1/2}\rightarrow x^\infty$ and $x^{k+1}\rightarrow x^\infty$}
\State Case (b)
\State $x^{k+1/2}$ and $x^{k+1}$ solution
\Else
\State Case (b), (c), (d), (e), (f), or (g).
\EndIf
\end{algorithmic}
\end{algorithm}

\begin{algorithm}
\caption{Feasibility test}
\label{alg2}
\begin{algorithmic}[0]
\State Parameters: $M$, $\varepsilon$, $z^0$
\For{$k=1,\dots$}
\State $x^{k+1/2}=P_K(z^k)$
\State $x^{k+1}=D(2x^{k+1/2}-z^k)+x_0$
\State $z^{k+1}=z^k+x^{k+1}-x^{k+1/2}$
\EndFor
\If{$\|z^k\|\ge M$ and $\|z^{k+1}-z^k\|>\varepsilon$}
\State Case (f)
\State Strictly separating hyperplane defined by $(z^{k+1}-z^k,(-v^Tx_0)/2)$
\ElsIf{$\|z^k\|\ge M$ and $\|z^{k+1}-z^k\|\le \varepsilon$}
\State Case (g)
\Else \,\,$\|z^k\|< M$
\State Case (a), (b), (c), (d), or (e)
\EndIf
\end{algorithmic}
\end{algorithm}

\begin{algorithm}
\caption{Boundedness test}
\label{alg3}
\begin{algorithmic}[0]
\State Prerequisite: \eqref{P} is feasible.
\State Parameters: $\gamma$, $M$, $\varepsilon$, $z^0$
\For{$k=1,\dots$}
\State $x^{k+1/2}=P_K(z^k)$
\State $x^{k+1}=D(2x^{k+1/2}-z^k)-\gamma Dc$
\State $z^{k+1}=z^k+x^{k+1}-x^{k+1/2}$
\EndFor
\If{$\|z^k\|\ge M$ and $\|z^{k+1}-z^k\|\ge \varepsilon$}
\State Case (d)
\State Improving direction $z^{k+1}-z^k$
\ElsIf{$\|z^k\|< M$}
\State Case (a), (b), or (c)
\Else
\State Case (a), (b), (c), or (e)
\EndIf
\end{algorithmic}
\end{algorithm}

\section{Numerical Experiments}
We test our algorithm on a library of weakly infeasible SDPs generated by \cite{LiuPataki2017_exact}.
These semidefinite programs are in the form:
\begin{equation*}
\begin{array}{ll}
\mbox{minimize}& C\bullet X\\
\mbox{subject to}& A_i\bullet X=b_i, i=1,...,m \\
\mbox{} &          X\in S^n_{+},
\end{array}
\end{equation*}
where $n=10$, $m=10$ or $20$, and 
$A\bullet B=\sum_{i=1}^{n}\sum_{j=1}^{n} A_{ij}B_{ij}$
denotes the inner product between two $n\times n$ matrices $A$ and $B$.

The library provides ``clean'' and ``messy'' instances.
Given a clean instance, a messy instance is created with
\begin{align*}
A_i&\leftarrow U^T(\sum_{j=1}^{m} T_{ij} A_j)U\,\, \text{for}\,\,i=1,...,m\\
b_i&\leftarrow \sum_{j=1}^{m}T_{ij}b_j\,\,\text{for}\,\,i=1,...,m,
\end{align*}
where $T\in \mathbb{Z}^{m\times m}$ and $U\in \mathbb{Z}^{n\times n}$ are random invertible matrices with entries in $[-2,2]$.

In \cite{LiuPataki2017_exact}, four solvers are tested, specifically, SeDuMi, SDPT3 and MOSEK from the YALMIP environment, and the preprocessing algorithm of Permenter and Parrilo \cite{PermenterParrilo2014_partial} interfaced with SeDuMi.
Table~\ref{Table.1} reports 
the numbers of instances determined infeasible out of 100 weakly infeasible instances. 
The four solvers have varying success in detecting infeasibility of the clean instances,
but none of them succeed in the messy instances.

\begin{table}[!htb]
\begin{minipage}{0.48\textwidth}
\caption{Percentage of infeasibility detection in \cite{LiuPataki2017_exact}}
\centering
\label{Table.1}
\begin{tabularx}{\textwidth}{lllll}
\hline\noalign{\smallskip}
          & \multicolumn{2}{l}{$m=10$} & \multicolumn{2}{l}{$m=20$} \\
          \noalign{\smallskip}
          \cline{2-3}\cline{4-5} 
          \noalign{\smallskip}
          & Clean        & Messy        & Clean        & Messy        \\
          \noalign{\smallskip}\hline\noalign{\smallskip}
SeDuMi    & 0            & 0            & 1            & 0            \\ 
SDPT3     & 0            & 0            & 0            & 0            \\ 
MOSEK     & 0            & 0            & 11           & 0            \\ 
PP+SeDuMi & 100          & 0            & 100            & 0          \\    
\noalign{\smallskip}\hline
\end{tabularx}
\end{minipage}
\hfill
\begin{minipage}{0.48\textwidth}
\caption{Percentage of infeasibility detection success}
\centering
\label{Table.2}
\begin{tabularx}{\textwidth}{lllll}
\hline\noalign{\smallskip}
& \multicolumn{2}{l}{$m=10$} & \multicolumn{2}{l}{$m=20$} \\ \noalign{\smallskip}\cline{2-5}\noalign{\smallskip}
& Clean        & Messy        & Clean        & Messy    \\ \noalign{\smallskip}\hline\noalign{\smallskip}
Proposed method  & 100        & 21   & 100    & 99  
\\ \noalign{\smallskip}\hline
\end{tabularx}
\vspace{0.5cm}
\caption{Percentage of success determination that problems are not strongly infeasible}
\label{Table.3}
\begin{tabularx}{\textwidth}{lllll}
\hline\noalign{\smallskip}
& \multicolumn{2}{l}{$m=10$} & \multicolumn{2}{l}{$m=20$} \\ \noalign{\smallskip}\cline{2-5}\noalign{\smallskip}
& Clean        & Messy        & Clean        & Messy    \\ \noalign{\smallskip}\hline\noalign{\smallskip}
Proposed method & 100          & 100          & 100          & 100          
\\ \noalign{\smallskip}\hline
\end{tabularx}
\end{minipage}
\end{table}

Our proposed method performs better.
However, it does require many iterations
 and does fail with some of the messy instances.
We run the algorithm with $N= 10^7$ iterations
and  label an instance infeasible if $1/\|z^N\|\le 8\times10^{-2}$
 (cf.\ Theorem~\ref{thm:feas} and \ref{thm:sfeas}).
Table~\ref{Table.2} reports
the numbers of instances determined infeasible out of 100 weakly infeasible instances. 

We would like to note that detecting whether or not a problem is strongly infeasible
is easier than detecting whether a problem is infeasible.
With $N=5\times 10^4$
and a tolerance of $\|z^N-z^{N+1}\|<10^{-3}$
(c.f\  Theorem~\ref{thm:sfeas})
our proposed method correctly determined
that all test instances are not strongly infeasible.
Table~\ref{Table.3} reports 
the numbers of instances determined not strongly infeasible out of 100 weakly infeasible instances.

\begin{acknowledgements}
W. Yin would like to thank Professor Yinyu Ye for his question regarding ADMM applied to infeasible linear programs during the 2014 Workshop on Optimization for Modern Computation held at Peking University.
\end{acknowledgements}

\bibliographystyle{spmpsci}
\bibliography{DRSCertificate}

\end{document}